\documentclass[11pt,a4paper]{amsart}

\usepackage{amsmath,amssymb,amscd}
\usepackage{a4wide}

\newtheorem{thm}{Theorem}

\newtheorem{cor}{Corollary}

\theoremstyle{definition}

\newtheorem{remark}{Remark}

\newcommand{\ts}{\hspace{0.5pt}}
\newcommand{\nts}{\hspace{-0.5pt}}

\newcommand{\R}{\mathbb{R}\ts}

\newcommand{\Z}{\mathbb{Z}}
\newcommand{\N}{\mathbb{N}}

\newcommand{\cD}{\mathcal{D}}
\newcommand{\cL}{\mathcal{L}}

\newcommand{\Per}{\mathrm{Per}}
\newcommand{\oplam}{\mbox{\Large $\curlywedge$}}
\newcommand{\exend}{\hfill $\Diamond$}

\begin{document}

\title{Toeplitz flows and model sets}

\author{M.~Baake}
\address{Faculty of Mathematics, Bielefeld University, Germany}
\email{mbaake@math.uni-bielefeld.de}

\author{T.~J\"ager}
\address{Institute of Mathematics, Friedrich Schiller University Jena, Germany}
\email{tobias.jaeger@uni-jena.de}

\author{D.~Lenz}
\address{Institute of Mathematics, Friedrich Schiller University Jena, Germany}
\email{daniel.lenz@uni-jena.de}


\begin{abstract}
  We show that binary Toeplitz flows can be interpreted as Delone
  dynamical systems induced by model sets and analyse the quantitative
  relations between the respective system parameters. This has a
  number of immediate consequences for the theory of model sets. In
  particular, we use our results in combination with special examples
  of irregular Toeplitz flows from the literature to demonstrate that
  irregular proper model sets may be uniquely ergodic and do not need
  to have positive entropy. This answers questions by Schlottmann and
  Moody.
\end{abstract}



\maketitle

\section{Introduction}

Toeplitz flows have played an important role in the development of
ergodic theory, since they provide a wide class of minimal dynamical
systems that may exhibit a variety of exotic properties. Accordingly,
they have often been employed to clarify fundamental questions on
dynamical systems and to provide examples for particular combinations
of dynamical properties, such as strict ergodicity and positive entropy
or minimality and absence of unique ergodicity \cite{JK69,MP79,Wil84}.

The aim of this article is to make this rich source of examples
available in the context of aperiodic order --- often referred to as
the mathematical theory of quasicrystals --- and more specifically
for the study of (general) cut and project sets. Cut and project
sets, introduced by Meyer \cite{Mey} in a somewhat different
context, are arguably the most important class of examples within
the theory of quasicrystals. With a focus on proper (but not
necessarily regular)  model sets their investigation has become a
cornerstone of the theory, see e.g. the survey papers
\cite{Moo1,Moo2}. Recently, also the more general  classes of
repetitive Meyer sets \cite{Auj,KS} (see the survey \cite{ABKL} as
well) and of weak model sets \cite{BHS,KR} have been studied in some
detail.

It is an elementary observation that any two-sided repetitive (or almost
periodic) sequence $\xi=(\xi_n)^{}_{n\in\Z}$ of symbols 0 nd 1 can be identified
with a Delone subset of $\Z$, for instance the one given by
$\cD(\xi)=\{n\in\Z\mid \xi_n=1\}$. Our main result, Theorem \ref{t.Toeplitz-CPS}
in Section \ref{section:main}, shows that if $\xi$ is a Toeplitz sequence, then
$\cD(\xi)$ can always be interpreted as a model set arising from a cut and
project scheme (CPS). Not surprisingly, the internal group of this CPS is chosen
as the odometer associated with the Toeplitz sequence.

Moreover, we provide a quantitative analysis and show that the
regularity and the scaling exponents of the Toeplitz sequence can be
computed either in terms of the measure of the boundary of the windows
in the CPS (in the case of irregular Toeplitz flows, Section
\ref{section:main}) or in terms of the box dimension of this boundary
(for regular Toeplitz flows, Section \ref{section:regular}). This ties
together the principal quantities of both system classes and
immediately allows to answer a number of questions on model sets,
which were -- to the best of our knowledge -- still open.

In particular, we thus obtain that a positive measure of the boundary
of the window of a cut and project set does imply neither positive
topological entropy, nor the existence of multiple ergodic measures
(Section \ref{section:irregular}).  Moreover, using Toeplitz examples
due to Downarowicz and Lacroix \cite{DL96}, we can then demonstrate
that model sets may have any given countable subgroup of $\R$ as their
dynamical spectrum, provided it contains infinitely many rationals. In
particular, irrational eigenvalues may occur despite the fact that the
underlying odometer has only a rational point spectrum, and this further
implies the existence of non-continuous eigenfunctions.

It is worth mentioning that the model sets provided by our
construction are minimal and  satisfy the  additional regularity
feature of properness. In particular, they fall into both the class
of repetitive Meyer sets and the class of weak model sets mentioned
above.

Our construction below deals with  binary Toeplitz sequences, as
this suffices to provide the desired counterexamples. However,
clearly, similar results  can be achieved    for Toeplitz sequences
over larger alphabets than $\{0,1\}$. Indeed, for the purposes of
the present paper it is not hard to reduce the case of arbitrary
Toeplitz sequences to binary sequences by identifying all letters
but one.

\section{Binary Toeplitz sequences and model
sets}\label{section:main}

Let us start by discussing Toeplitz sequences and flows. For background and
references we refer the reader to \cite{Do05}.  Let $\varSigma=\{0,1\}^\Z$ and
denote by $\sigma \! : \, \varSigma\xrightarrow{\quad}\varSigma$ the left shift.
Suppose $\xi\in\varSigma$ is a Toeplitz sequence, which means that for all
$k\in\Z$ there exists $p\in\N$ such that $\xi_{k+np}=\xi_k$ for all
$n\in\Z$. Then, the shift orbit closure
$\varSigma_\xi=\overline{\{\sigma^n(\xi)\mid n\in\Z\}}$ is a minimal set for the
shift $\sigma$, and $(\varSigma_\xi,\sigma)$ is called the \emph{Toeplitz flow}
generated by $\xi$. Elements of $\varSigma_\xi$ that are not Toeplitz sequences
themselves are called \emph{Toeplitz orbitals}. The latter necessarily exist in
any $\varSigma_\xi$ that is built from a non-periodic Toeplitz sequence $\xi$;
cf.\ \cite[Cor.~4.2]{TAO}. In the remainder of this paper, we shall restrict
our attention to non-periodic Toeplitz sequences.

In line with the standard literature, we call
\[
  \Per(p,\xi) \, = \, \{k\in\Z\mid \xi_k=\xi_{k+np} \;\:
   \mbox{for all}\;\:  n\in\Z\}
\]
the \emph{$p$-skeleton} of $\xi$ and refer to its elements as
\textit{$p$-periodic positions} of $\xi$. A $p\in\N$ is an
\emph{essential period} of $\xi$ if $\Per(p',\xi)\neq\Per(p,\xi)$ for
all $p'<p$. A \emph{period structure} for $\xi$ is a sequence
$(p^{}_\ell)^{}_{\ell\in\N}$ of essential periods such that, for all
$\ell\in\N$, $p^{}_\ell$ is an essential period of $\xi$ that divides
$p^{}_{\ell+1}$ and that, together, they satisfy $\bigcup_{\ell\in\N}
\Per(p^{}_\ell,\xi)=\Z$. Such a sequence always exists and can be
obtained, for example, by defining $p^{}_\ell$ as the multiple of all
essential periods occurring for the positions in $[-\ell, \ldots ,
\ell]$. The \emph{density} of the $p$-skeleton is defined as
\[
   D(p) \, = \, \# \bigl(\Per(p,\xi)\cap [0,p-1]\bigr)/p \ts .
\]
A Toeplitz sequence and the associated flow are called
\emph{regular}, if $\lim_{\ell\to\infty} D(p^{}_\ell)=1$ and
\emph{irregular} otherwise. This distinction turns out to be
independent of the choice of the period structure.

Given a period structure $(p^{}_\ell)_{\ell\in\N}$, we let $q^{}_\ell
=p^{}_\ell/p^{}_{\ell-1}$ for $\ell\geq 1$, with the convention that
$p^{}_0=1$. Then, the compact Abelian group
$\varOmega=\prod_{\ell\in\N} \Z / {q^{}_\ell} \Z$ equipped with the
addition defined according to the carry over rule is called the
\emph{odometer group} with \emph{scale}
$(q^{}_\ell)^{}_{\ell\in\N}$. We denote the Haar measure on
$\varOmega$ by $\mu$ and let $\tau \! : \, \varOmega
\xrightarrow{\quad}\varOmega$ with $\omega\mapsto\omega+(1,0,0,
\ldots)$ denote the canonical minimal group rotation on
$\varOmega$. We call $(\varOmega,\tau)$ the \emph{odometer
  associated{\ts\ts}\footnote{We note that period structures of
    Toeplitz sequences are not uniquely defined, but all the odometers
    with scales corresponding to different period structures of the
    same Toeplitz sequence are isomorphic.}  to
  $(\varSigma_\xi,\sigma)$}.  This odometer $(\varOmega,\tau)$
coincides with the maximal equicontinuous factor (MEF) of
$(\varSigma_\xi,\sigma)$, and the factor map $\beta \!  :
\,\varSigma_\xi\xrightarrow{\quad}\varOmega$ can be defined by
\[
    \beta(x)\, = \, \omega
    \quad :\Longleftrightarrow \quad
    \Per\bigl(p^{}_\ell,\sigma^{k(\ell,\omega)}(x)\bigr)
     \, = \, \Per(p^{}_\ell,\xi)\ts ,
     \quad \mbox{for all}\;\: \ell\in\N \ts ,
\]
where $k(\ell,\omega)=\sum_{i=1}^\ell \omega_ip_{i-1}$; cf.\
\cite{DL96} for details.  Given $w\in\prod_{i=1}^\ell \Z / {q_i} \Z$,
we also let $k(\ell,w)=\sum_{i=1}^\ell w_ip_{i-1}$.

\begin{remark}\label{rem:alternative}
  There is an alternative equivalent description of the odometer,
  which is actually closer to considerations in the quasicrystal
  literature, as e.g. in \cite{BM}. As this is instructive in the
  context of our construction below we shortly discuss this: Let
  $\varOmega'$ be the inverse limit of the system $(\Z /
  {p^{}_\ell}\Z)^{}_{\ell \geq 0}$, so the elements of $\varOmega'$
  are the sequences $(x^{}_\ell)^{}_{\ell \geq 0}$ with $ x^{}_\ell
  \in \Z / {p^{}_\ell} \Z$ and $x^{}_{\ell -1} = \pi^{}_{\ell}
  (x^{}_\ell)$ for each $\ell\in \N$. Here, $\pi^{}_\ell \! : \, \Z /
  {p^{}_\ell} \Z \longrightarrow \Z / {p^{}_{\ell -1}}\Z $ is the
  canonical projection. Then, $\varOmega'$ is an Abelian group under
  componentwise addition, and the map
\[
   \varOmega \xrightarrow{\quad} \varOmega^{\ts \prime},\;
   \omega \mapsto ( k(\ell,\omega))^{}_{\ell \geq 0} \ts ,
\]
provides an isomorphism of topological groups. Under this
isomorphism, the map $\tau$ on $\varOmega$ corresponds to addition of
$1$ in each component of elements of $\varOmega^{\ts \prime}$.  \exend
\end{remark}

Let us now turn to CPSs and model sets, where
we refer the reader to \cite{TAO} and references therein for background and
general notions. In general, a CPS $(G,H,\cL)$ is given by a pair of
locally compact Abelian groups $G,H$ together with a discrete
co-compact subgroup $\cL$ of $G\times H$ such that $\pi:G\times H\to
G$ is injective on $\cL$ and $\pi_\mathrm{int}:G\times H\to H$ maps
$\cL$ to a dense subset of $H$. Given any subset ({\em window}) $W$ of
$ H$, such a CPS produces a subset of $G$, called a \emph{cut and
  project set}, given by
\[
   \oplam(W)\ = \ \pi\left((G\times W)\cap \cL\right) \ts .
\]
Such a set is called a \emph{model set} when $W$ is relatively compact
with non-empty interior. In this case, $\oplam(W)$ is always a Delone
set. When, in addition, the boundary $\partial W$ of the window has
zero measure in $H$, the model set is called \emph{regular}. As a
standard case, one considers compact windows $W$ that satisfy
$\varnothing \ne W = \overline{\mathrm{int}(W)}$, in which case they
are called \emph{proper}.

Since we consider Toeplitz sequences as weighted subsets of $\Z$, the
easiest way to describe them as model sets works for a CPS of the form
$(G,H,\cL)$ where $G=\Z$ or $\R$ and $H = \varOmega$ is the odometer
from above. So, we consider the situation summarised in the following
diagram,
\begin{equation*}
\renewcommand{\arraystretch}{1.2}\begin{array}{r@{}ccccc@{}l}
   & G & \xleftarrow{\,\;\;\pi\;\;\,} & \Z \times H &
        \xrightarrow{\;\pi^{}_{\mathrm{int}\;}\,} & H & \\
   & \cup & & \cup & & \cup & \hspace*{-1ex}
   \raisebox{1pt}{\text{\footnotesize dense}} \\
   & \Z & \xleftarrow{\, 1-1 \,} & \cL &
   \xrightarrow{\; \hphantom{1-1} \;} & \pi^{}_{\mathrm{int}}(\cL) & \\
   & \| & & & & \| & \\
   & L & \multicolumn{3}{c}{\nts\xrightarrow{\qquad\qquad\;\star
       \;\qquad\qquad}}
       &  {L_{}}^{\star\nts} & \\
\end{array}\renewcommand{\arraystretch}{1}
\end{equation*}
Further, $\cL$ is a lattice in $G\times H$ that emerges as a diagonal
embedding,
\begin{equation}\label{eq:def-lat}
    \cL \, := \, \{ (n, n^{\star} ) \mid n\in\Z\} \ts ,
\end{equation}
where $n^{\star} := \tau^{n} (0)$ defines the so-called $\star$-map
$\star \! : \, G=\Z \xrightarrow{\quad} H$. Clearly, the restriction
of $\pi$ to $\cL$ is one-to-one and the restriction of
$\pi^{}_{\mathrm{int}}$ has dense range. Using the $\star$-map, a
\emph{cut and project set} for $(G,H,\cL)$ and window $W$ can equally
be written as
\[
     \oplam (W) \, = \, \{ x \in L \mid x^{\star} \in W \} \ .
\]
Using these ingredients, our main result now reads as follows.
\begin{thm}\label{t.Toeplitz-CPS}
  Let\/ $\varSigma=\{0,1\}^\Z$ and suppose\/ $\xi\in\varSigma$ is a
  non-periodic Toeplitz sequence with period structure\/
  $(p^{}_\ell)^{}_{\ell\in\N}$. Let\/ $(\varOmega,\tau)$ be the
  associated odometer. Then,
\[
      \cD(\xi) \, = \, \{n\in\Z \mid \xi_n=1 \}
\]
is a model set for the CPS\/ $(\Z, \varOmega, \cL)$ or\/ $(\R,
\varOmega, \cL)$, with the lattice\/ $\cL$ of equation\eqref{eq:def-lat}
and the\/ $\star$-map defined above.  Moreover, the corresponding
window\/ $W\subseteq\varOmega$ is proper and satisfies
\[
      \mu(\partial W) \, = \,
     1-\lim_{\ell\to\infty} D(p^{}_\ell) \ts .
\]
\end{thm}

\begin{proof}
  We discuss the case $(\Z, \varOmega, \cL)$ with $\cL$ as in
  \eqref{eq:def-lat}; the other case is analogous, because we view
  $\cD(\xi)$ as a subset of $\Z$, so that the $\R$-action emerges from
  the $\Z$-action by a simple suspension with a constant height
  function.  In order to derive the window $W$ for the CPS, we denote
  cylinder sets in $\varOmega$ either by $[w]=[w^{}_1, \ldots ,
  w^{}_\ell]=\{\omega\in\varOmega\mid \omega_i=w_i \text{ for } 1
  \leqslant i\leqslant \ell\}$ with $w\in\prod_{i=1}^\ell \Z / {q_i}
  \Z$ or, given $\omega\in\varOmega$ and $\ell\in\N$, by
  $[\omega]^{}_\ell =[\omega^{}_1, \ldots ,\omega^{}_\ell]$. Note that
  $\mu \bigl([w]\bigr)=1/p^{}_\ell$.

Consider
\[
  A(\ell,s)\, = \, \Bigl\{ w\in\prod_{i=1}^\ell \Z / {q_i}\Z
  \, \Big| \,
  k(\ell,w) \text{ is a } p^{}_{\ell} \text{-periodic position of } \xi
  \text{ and } \xi_{k(\ell,w)} = s \Bigr\}
\]
with $s\in \{0,1 \}$. Then, define $U_\ell=\bigcup_{w\in A(\ell,1)}
[w]$ and $V_\ell=\bigcup_{w\in A(\ell,0)} [w]$. Clearly,
$U_\ell\subset U_{\ell +1}$ and $V_\ell\subset V_{\ell +1}$ hold for
any $\ell$.  Set $U=\bigcup_{\ell\in\N}U_\ell$ and
$V=\bigcup_{\ell\in\N}V_\ell$. Now, we let $W=\overline{U}$ and claim
that this window $W$ satisfies the assertions of our theorem.

First, we show  that our CPS $(\Z,\varOmega,\cL)$ together with the
window $W$ produces the Delone set $\cD(\xi)$ as its model set, that
is,
\[
   \oplam(W)\, := \, \{ n\in\Z \mid \tau^n(0)\in W\}
   \, = \, \cD(\xi) \ts .
\]
In order to do so, fix $k\in\Z$ and suppose $\xi_k=1$, so that
$k\in\cD(\xi)$. Let $\ell$ be the least integer such that $k$ is a
$p^{}_\ell$-periodic position of $\xi$. Then, there exists a unique
$k'\in[0,p^{}_\ell-1]$ such that $k'=k+n \ts p^{}_\ell$ for some
$n\in\Z$.  This $k'$ is a $p^{}_\ell$-periodic position as well and we
also have $\xi_{k'}=1$. Further, we have $k'=k(\ell,w)$ for a unique
$w\in \prod_{i=1}^\ell \Z/{q_i}\Z$. However, this means that we have
$[w]\subseteq U\subseteq W$ by construction. Since $\tau^m(0)\in[w]$
for all $m\in k(\ell,w)+p^{}_\ell\Z$ (note that $\tau^{k(\ell,w)} (0)
=(w,0,0,\ldots)$ and any cylinder of length $\ell$ is
$p^{}_\ell$-periodic for $\tau$), we in particular have that
$\tau^k(0)\in W$, so that $k\in \oplam(W)$. In a similar way, we
obtain that $\xi_k=0$ implies $\tau^k(0)\in V$ and thus $k\notin
\oplam(W)$ (note here that $V$ is open, so that $V\subseteq
\varOmega\setminus W$). This proves $\oplam(W)=\{k\in\Z\mid
\xi_k=1\}=\cD(\xi)$.

Next, we determine the measure of $\partial W$ (and  obtain
properness as a byproduct). We have
\[
   \# \bigl(A(\ell,0)\cup A(\ell,1)\bigr)\, = \,
   \#\{k\in [0,p^{}_\ell-1]\mid
   k \textrm{ is a } p^{}_\ell\textrm{-periodic position}\}
   \, = \, D(p^{}_\ell)\cdot p^{}_\ell \ts .
\]
This means that $\mu(U_\ell\cup V_\ell) = D(p^{}_\ell)$, and thus
\[
   \mu(U\cup V)\, = \lim_{\ell\to\infty} \mu(U_\ell\cup V_\ell)
   \, = \lim_{\ell\to\infty} D(p^{}_\ell) \ts .
\]
Thus, it suffices to show that $\partial W=\varOmega\setminus (U\cup
V)$. By openness and disjointness of $U$ and $V$, this is equivalent to
$(W= \, ) \, \overline{U}= \varOmega\setminus V$ and $\overline{V}
=\varOmega\setminus U$. As the situation is symmetric, we restrict to
prove $W=\varOmega\setminus V$. The inclusion $W\subset \varOmega
\setminus V$ is clear. It remains to show the opposite inclusion. To
that end, fix $\omega\in \varOmega\setminus V$ and $\kappa\in \N$.  We
are going to show that $U$ intersects every cylinder neighbourhood
$[\omega]_\kappa$ of $\omega$, so that $\omega\in \overline{U}=W$.

As $\bigcup_{\ell\in\N} \Per(p^{}_\ell,\omega)=\Z$, there exists a
least integer $\ell$ such that $k=k(\kappa,\omega_1, \ldots
,\omega_\kappa)$ is a $p^{}_\ell$-periodic position. First, suppose
that $\ell\leq \kappa$. Then, $k$ is a $p_\kappa$-periodic position,
and we have $[\omega]_\kappa\subseteq U_\kappa$ if $\xi^{}_k=1$ and
$[\omega]_\kappa\subseteq V_\kappa$ if $\xi^{}_{k}=0$. The latter is
not possible, since we assume $\omega\notin V$. Hence, we have that
$[\omega]_\kappa\subseteq U$.  Secondly, suppose that $\ell
>\kappa$. Then, $k$ cannot be a $p_\kappa$-periodic position, and
hence there exists $n\in\Z$ such that $k'=k+np_\kappa$ satisfies
$\xi_{k'}=1$. Choose the least $\ell'\in\N$ such that $k'$ is a
$p^{}_{\ell'}$-periodic position and let $v\in\prod_{i=1}^{\ell'}
\Z/{q_i}\Z$ be such that $k'=k(\ell',v)\bmod p^{}_{\ell'}$. Set
$k''=k(\ell',v)$. By construction, we have $[v]\subseteq
U_{\ell'}\subseteq U$. At the same time, we have $v_i=\omega_i$ for
all $1 \leqslant i \leqslant \kappa$, since
$k''=k+np_\kappa+mp^{}_{\ell'}$ for some $m\in\Z$. Hence, we have that
$[v]\subseteq [\omega]_\kappa$ and thus $U\cap [\omega]_\kappa\neq
\varnothing$.

As mentioned already, the argument is completely symmetric with
respect to $U$ and $V$, and we also obtain
$\overline{V}=\varOmega\setminus U$. This then  implies that
$U=\mathrm{int} (W)$, and since $W=\overline{U}$ by definition, we
obtain that $W$ is proper.
\end{proof}

\section{Regular Toeplitz flows}\label{section:regular}

For the case of \emph{regular} Toeplitz sequences, more information is
available to relate the scaling behaviour of $D(p^{}_\ell)$ to the box
dimension of the boundary of the window.  In order to state the
result, we assume that $d$ is a metric on $\varOmega$ that generates
the product topology and is invariant under the group rotation
$\tau$. Note that since cylinder sets in $\varOmega$ are mapped to
cylinder sets, and $\tau$ is transitive on $\varOmega$, all cylinder
sets of a given level $\ell$ have the same diameter $d_\ell$. The
choice of the sequence $d_\ell$ defines the metric and is more or less
arbitrary, as long as it is decreasing in $\ell$. The box dimension of
$\varOmega$ depends on this choice and is given by
\[
   \mathrm{Dim}^{}_B(\varOmega) \, = \ts \lim_{\ell\to\infty}
   \frac{\log (p^{}_\ell)}{ \log (d^{}_\ell)} \ts .
\]
If this limit does not exist, then one defines upper and lower box
dimension $\overline{\mathrm{Dim}}_B(\varOmega)$ and
$\underline{\mathrm{Dim}}_B(\varOmega)$ by using the limit superior,
respectively inferior. We also note that the canonical choice for the
metric $d$ is given by $d^{}_\ell=p^{-1}_{\ell+1}$, but our statement is
valid in general.

\begin{thm}\label{t.box-dimension}
  Suppose that, in the situation of Theorem~$\ref{t.Toeplitz-CPS}$, we
  have\/ $\lim_{\ell\to\infty} D(p^{}_\ell)=1$. Then, the window\/ $W$
  can be chosen such that
\begin{align*}
  \overline{\mathrm{Dim}}^{}_{B}(\partial W) \, & = \,
  \left(1+\varlimsup_{\ell\to\infty}
  \frac{\log(1-D(p^{}_\ell))}{\log (p^{}_\ell)}\right)
  \overline{\mathrm{Dim}}^{}_B(\varOmega)
\intertext{and}
  \underline{\mathrm{Dim}}^{}_{B}(\partial W) \, & = \,
  \left(1+\varliminf_{\ell\to\infty}
  \frac{\log(1-D(p^{}_\ell))}{\log (p^{}_\ell)}\right)
  \underline{\mathrm{Dim}}^{}_B(\varOmega) \ts .
\end{align*}

\end{thm}

\begin{proof}
  The construction in the proof of Theorem~\ref{t.Toeplitz-CPS} is completely
  independent of the value of $\lim_{\ell\to\infty} D(p^{}_\ell)$, and in
  particular applies also to the regular case. Hence, we may choose the same
  window $W$ as above and need only to determine the box dimension of $\partial
  W$.

To that end, note that $\partial W=\varOmega\setminus(U\cup V)$ and
$U_\ell\cup V_\ell\subseteq (U\cup V)$, so that $\varOmega\setminus
(U_\ell\cup V_\ell)$ contains $\partial W$. However, we have that
\[
  \varOmega\setminus(U_\ell\cup V_\ell) \ = \
  \bigcup_{w\in\Z^\ell\setminus A(\ell,0)\cup A(\ell,1)} [w] \ ,
\]
so that this set is a union of $N(\ell)=(1-D(p^{}_\ell))\cdot
p^{}_\ell$ cylinders of order $\ell$. Moreover, it is not possible to
cover $\partial W$ with a smaller number of such cylinders, so that
$N(\ell)$ is the least number of sets of diameter $d_\ell$ needed to
cover $\partial W$. Hence, we obtain
\[
  \overline{\mathrm{Dim}}_B(\partial W) \ = \
  \varlimsup_{\ell\to\infty}
  \frac{\log((1-D(p^{}_\ell))\cdot p^{}_\ell)}{\log d_\ell} \ =
  \ \left(1+\varlimsup_{\ell\to\infty}
  \frac{\log(1-D(p^{}_\ell))}{\log p^{}_\ell}\right)
 \cdot \overline{\mathrm{Dim}}_B(\varOmega) \ .
\]
The analogous computation yields the relation for the lower box
dimensions.
\end{proof}

\begin{remark}
Let us point out some further properties and directions as follows.

(a) As a consequence of our main theorem, a Toeplitz sequence is
regular if and only if the associated model set is regular.

(b) As  is well-known, regular Toeplitz flows are almost one-to-one
extensions of their MEF; cf.\ \cite{DL96}. In particular, they
are uniquely ergodic and have pure point dynamical spectrum with
continuous eigenfunctions, which separate almost all
points. Therefore, the general characterisation of dynamical systems
coming from regular models sets provided in \cite{BLM} directly
applies to provide a model set construction for these systems.

(c)  As  regular Toeplitz flows have pure point dynamical spectrum,
they also exhibit pure point diffraction by the general equivalence
theorem, see \cite{BL} and references therein for details and
background. As they are uniquely ergodic, each individual sequence of
the flow then exhibits the same pure point diffraction. So, the
method of \cite{BM} can be used to provide a CPS for a regular
Toeplitz sequence. This leads to an alternative, but equivalent,
way in the spirit of Remark~\ref{rem:alternative}, as can be seen
from the example of the period doubling chain in \cite{BM}.
\exend
\end{remark}

\section{Consequences for irregular Toeplitz
flows}\label{section:irregular}

There has been quite some speculation on connections between
irregularity of the model set and occurrence of positive entropy or
failure of unique ergodicity for the associated dynamical systems.
Indeed,  Schlottmann  asks whether irregularity of the model set
implies failure of unique ergodicity \cite{Schlottmann} and Moody
has suggested that irregularity is related to positive entropy. The
suggestion of Moody is recorded in  \cite{Pleasants} (later subsumed
in \cite{PleasantsHuck}) and is also discussed in  the introduction
to \cite{HR}. When combined with examples of Toeplitz systems
studied in the past, our main theorem allows us to answer these
speculations by presenting model sets with various previously
unknown features (such as irregularity combined with unique
ergodicity and zero entropy). This is discussed in this section. Let
us emphasise that all Toeplitz flows are minimal and that the model
sets presented below are even proper (as Theorem
\ref{t.Toeplitz-CPS} provides proper model sets).

\smallskip

Arguably among the most interesting examples in our present context
are irregular Toeplitz flows with zero entropy \cite{Ox52}, as these
immediately imply the following statement.
\begin{cor}
  Positive measure of the boundary of a window of a CPS is not a
  sufficient criterion for positive topological entropy of the
  associated Delone dynamical system.  \qed
\end{cor}
We note that there even exist irregular Toeplitz flows (and thus
irregular model sets) whose word complexity is only linear
\cite{GJ15}.  A number of further interesting examples are available
in the literature. Amongst these are the following.
 \begin{itemize}
 \item Irregular Toeplitz flows may be uniquely ergodic \cite{Wil84},
   so that a window for a CPS with a boundary of positive measure does
   \emph{not} contradict unique ergodicity of the resulting Delone
   dynamical system. Moreover, the set of ergodic invariant measures
   of an irregular Toeplitz flow may have any cardinality.
 \item Any countable subgroup of $\R$ that contains infinitely many
   rationals can be the dynamical spectrum of a Toeplitz flow
   \cite{DL96}, and thus of the Delone dynamical system arising from a
   CPS. As all continuous eigenvalues of a Toeplitz flow are rational,
   this gives in particular examples of model sets with (many)
   measurable eigenvalues (see \cite{DFM} as well for a recent study
   of Toeplitz systems of finite topological rank with measurable
   eigenvalues).
 \item Oxtoby's original example of a Toeplitz flow with zero entropy
   in \cite{Ox52} is not uniquely ergodic, and the same is true of the
   example in \cite{GJ15}. However, there also exist uniquely ergodic
   irregular Toeplitz flows both with and without positive entropy
   \cite{Wil84}. In particular, there exist irregular model sets with
   uniquely ergodic minimal dynamical systems with zero entropy.
 \end{itemize}

\begin{remark}
  There is an emerging theory of weak model sets (see, for instance,
  \cite{BHS,HR,KR,JLO}) dealing with irregular windows for a given CPS. If
  the arising model sets satisfies a maximality condition for its
  density, then the associated dynamical systems will generally not be
  minimal. Thus, the irregular Toeplitz sequences are, in this sense,
  never weak model sets of maximal density. They rather provide a
  versatile class of examples to explore the possibilities that emerge
  from missing out on the maximality property.  \exend
\end{remark}
\bigskip

\section*{Acknowledgements}

The presented ideas originated from discussions during the workshop
\emph{Dynamical Systems and Dimension Theory}, 8--12 September 2014,
W\"oltingerode, Germany. This was supported by the `Scientific
Network: Skew product dynamics and multifractal analysis' (DFG-grant
Oe 538/3-1). TJ acknowledges support by the Emmy-Noether program and
the Heisenberg program of the DFG (grants Ja 1721/2-1 and Oe 538/6-1).
Further, this work was also supported by the German Research
Foundation (DFG), within the CRC 701.

\end{document}